\documentclass{amsart}

\usepackage[all]{xy}
\usepackage{times,fullpage}

\usepackage{xr-hyper}

\usepackage{hyperref}

\theoremstyle{plain}
\newtheorem{theorem}[subsection]{Theorem}

\newtheorem{lemma}[subsection]{Lemma}
\newtheorem{corollary}[subsection]{Corollary}

\theoremstyle{definition}

\newtheorem{example}[subsection]{Example}

\newtheorem{construction}[subsection]{Construction}
\newtheorem{notation}[subsection]{Notation}
\newtheorem*{conventions*}{Conventions}
\newtheorem{remark}[subsection]{Remark}

\numberwithin{equation}{subsection}
\newcommand{\colim}{\mathop{\mathrm{colim}}}

\begin{document}

\title{Crystalline cohomology and de Rham cohomology}

\author{Bhargav Bhatt}
\author{Aise Johan de Jong}

\begin{abstract}
The goal of this short paper is to give a slightly different perspective
on the comparison between crystalline cohomology and de Rham cohomology.
Most notably, we reprove Berthelot's comparison result without using
pd-stratifications, linearisations, and pd-differential operators.
\end{abstract}

\maketitle

\noindent
Crystalline cohomology is a $p$-adic cohomology theory for varieties in
characteristic $p$ created by Berthelot \cite{Berthelot}. It was designed
to fill the gap at $p$ left by the discovery \cite{SGA4tome3} of
$\ell$-adic cohomology for $\ell \neq p$. The construction of crystalline
cohomology relies on the {\em crystalline} site, which is a better behaved
positive characteristic analogue of Grothendieck's {\em infinitesimal}
site \cite{GrothendieckCrystals}. The motivation for this definition 
comes from Grothendieck's theorem \cite{GrothendieckAlgdeRham} identifying 
infinitesimal cohomology of a complex algebraic variety with its singular 
cohomology (with $\mathbf{C}$-coefficients); in particular, infinitesimal 
cohomology gives a purely algebraic definition of the ``true'' cohomology 
groups for complex algebraic varieties. The fundamental structural result of
Berthelot \cite[Theorem V.2.3.2]{Berthelot} is a direct $p$-adic analogue
of this reconstruction result: the crystalline cohomology of a smooth
$\mathbf{F}_p$-variety $X$  is identified with the
de Rham cohomology of a lift of $X$ to $\mathbf{Z}_p$, provided one exists.
In particular, crystalline cohomology produces the ``correct'' Betti
numbers, at least for liftable smooth projective varieties (and, in fact,
even without liftability by \cite{KatzMessing}).  We defer to
\cite{IllusieCrysCoh} for a detailed introduction, and connections
with $p$-adic Hodge theory.

\medskip\noindent
Our goal in this note is to give a different perspective on the relationship
between de Rham and crystalline cohomology. In particular, we give a
short proof of the aforementioned comparison result
\cite[Theorem V.2.3.2]{Berthelot}; see Theorem \ref{theorem-global-cohomology}.
Our approach replaces Berthelot's differential methods (involving
stratifications and linearisations) with a resolutely
\v{C}ech-theoretic approach. It seems that
Theorem \ref{theorem-global}
is new, although it may have been known to experts in the field. This theorem 
also appears in forthcoming work of Alexander Beilinson \cite{Beilinson}.

\begin{conventions*}
Throughout this note,  $p$ is a fixed prime number.
Our base scheme will be typically be $\Sigma = \mathrm{Spec}(\mathbf{Z}_p)$, though occasionally we discuss the theory over  $\Sigma_e = \mathrm{Spec}(\mathbf{Z}_p/p^e)$ as well (for some $e \geq 1$).
All divided powers will be compatible with the divided powers
on $p\mathbf{Z}_p$.
Modules of differentials on divided power algebras are compatible
with the divided power structure.
A general reference for divided powers and the crystalline site
is \cite{Berthelot}.
\end{conventions*}

\section{Review of modules on the crystalline site}
\label{section-modules}

\noindent
Let $S$ be a $\Sigma$-scheme such that $p$ is locally nilpotent on $S$.
The (small) crystalline site of $S$ is denoted $(S/\Sigma)_{cris}$. Its objects 
are triples $(U,T,\delta)$ where $U \subset S$ is an open subset, 
$U \subset T$ is a nilpotent thickening of $\Sigma$-schemes, and $\delta$ is
a divided power structure on the ideal of $U$ in $T$; the morphisms are the obvious ones, 
while coverings of $(U,T,\delta)$ are induced by Zariski covers of $T$.
The structure sheaf 
$\mathcal{O}_{S/\Sigma}$ of $(S/\Sigma)_{cris}$ is defined by 
$\mathcal{O}_{S/\Sigma}( (U,T,\delta) )= \Gamma(T,\mathcal{O}_T)$.

\medskip\noindent
Given a $\mathbf{Z}_p/p^e$-algebra $B$ and an ideal $J \subset B$ endowed with
divided powers $\delta$, the module of {\em differentials compatible with
divided powers} is the quotient of the module of $\Sigma$-linear differentials 
 by the relations $\text{d}\delta_n(x) = \delta_{n-1}(x)\text{d}(x)$, for
$x \in J$ and $n \geq 1$.
We simply write $\Omega^1_B$ for this module as confusion is unlikely.  The formation of $\Omega^1_B$ commutes with 
localisation on $B$, so the formula
$\Omega^1_{S/\Sigma}( (U,T,\delta) )= \Gamma(T,\Omega^1_T)$ defines a sheaf
$\Omega^1_{S/\Sigma}$ on $(S/\Sigma)_{cris}$.
Like its classical analogue, the sheaf $\Omega^1_{S/\Sigma}$ can also be described via the diagonal as follows. Given an 
object $(U, T, \delta)$ of $(S/\Sigma)_{cris}$, let $(U, T(1), \delta(1))$ be the 
product of $(U, T, \delta)$ with itself in $(S/\Sigma)_{cris}$: the scheme $T(1)$ is simply the divided power envelope of $U \subset T \times_{\Sigma} T$, with $\delta(1)$ being the induced divided power structure. The diagonal 
map $\Delta : T \to T(1)$ is a closed immersion corresponding to an ideal sheaf $\mathcal{I}$ with divided powers,
and we have
$$
\Omega^1_{S/\Sigma}( (U,T,\delta) ) = \Gamma(T, \mathcal{I}/\mathcal{I}^{[2]}),
$$
where $\mathcal{I}^{[2]}$ denotes the second divided power of $\mathcal{I}$. 
For $i \geq 0$ we define $\Omega^i_{S/\Sigma}$ as the $i$-th exterior
power of $\Omega^1_{S/\Sigma}$.

\medskip\noindent
An $\mathcal{O}_{S/\Sigma}$-module $\mathcal{F}$ on
$(S/\Sigma)_{cris}$ is called {\it quasi-coherent} if for every object
$(U, T, \delta)$, the restriction $\mathcal{F}_T$ of $\mathcal{F}$
to the Zariski site of $T$ is a quasi-coherent $\mathcal{O}_T$-module.
Examples include $\mathcal{O}_{S/\Sigma}$ and $\Omega^i_{S/\Sigma}$ for all $i > 0$.

\medskip\noindent
An $\mathcal{O}_{S/\Sigma}$-module $\mathcal{F}$ on
$(S/\Sigma)_{cris}$ is called a {\it crystal in quasi-coherent modules}
if it is quasi-coherent and for every morphism
$f : (U, T, \delta) \to (U', T', \delta')$ the comparison map
$$
c_f : f^*\mathcal{F}_{T'} \to \mathcal{F}_T
$$
is an isomorphism. For example, the sheaf $\mathcal{O}_{S/\Sigma}$ is a crystal
(by fiat), but the sheaves $\Omega^i_{S/\Sigma}$, $i > 0$ are
{\bf not} crystals.

\medskip\noindent
Given a crystal $\mathcal{F}$ in quasi-coherent modules and an object $(U, T, \delta)$, 
the projections define canonical isomorphisms
$$
\text{pr}_1^*\mathcal{F}_T \xrightarrow{c_1}
\mathcal{F}_{T(1)} \xleftarrow{c_2} 
\text{pr}_2^*\mathcal{F}_T.
$$
These comparison maps are functorial in the objects of the crystalline site.
Hence we obtain a canonical map
\begin{equation}
\label{equation-connection}
\nabla : \mathcal{F} \longrightarrow
\mathcal{F} \otimes_{\mathcal{O}_{S/\Sigma}} \Omega^1_{S/\Sigma}
\end{equation}
such that for any object $(U, T, \delta)$ and any section
$s \in \Gamma(T, \mathcal{F}_T)$ we have
$$
c_1(s \otimes 1) - c_2(1 \otimes s)= \nabla(s) \in \mathcal{I}/\mathcal{I}^{[2]} \otimes_{\mathcal{O}_{S/\Sigma}} \mathcal{F}_{T(1)}.
$$
Transitivity of the comparison maps implies this connection is
integrable, hence defines a {\it de Rham complex}
$$
\mathcal{F} \to
\mathcal{F} \otimes_{\mathcal{O}_{S/\Sigma}} \Omega^1_{S/\Sigma} \to
\mathcal{F} \otimes_{\mathcal{O}_{S/\Sigma}} \Omega^2_{S/\Sigma} \to \cdots.
$$
We remark that this complex does not terminate in general.

\section{The de Rham-crystalline comparison for affines}
\label{section-affine}

\noindent
In this section, we discuss the relationship between de Rham and crystalline cohomology (with coefficients) when $S$ is affine. First, we establish some notation that will be used throughout this section.

\begin{notation}
\label{notation-affine}
Assume $S = \mathrm{Spec}(A)$ for a $\mathbf{Z}/p^N$-algebra $A$ (and some $N > 0$). Choose a polynomial algebra $P$ over $\mathbf{Z}_p$ 
and a surjection $P \to A$ with kernel $J$.
Let $D = D_J(P)^\wedge$ be the $p$-adically completed
divided power envelope of $P \to A$. We set $\Omega^1_D := \Omega^1_P \otimes^\wedge_P D$, so $\Omega^1_D/p^e \simeq \Omega^1_{D/p^eD}$. We also set $D(0) = D$ and let
$$
D(n) =
D_{J(n)}(P \otimes_{\mathbf{Z}_p} \cdots \otimes_{\mathbf{Z}_p} P)^\wedge
$$
where $J(n) = \text{Ker}(P \otimes \cdots \otimes P \to A)$
and where the tensor product has $(n + 1)$-factors. For each $e \geq N$ and any $n \geq 0$, we have a natural object $(S,\mathrm{Spec}(D(n)/p^e(n)),\delta(n))$ of $(S/\Sigma)_{cris}$. Using this, for an abelian sheaf $\mathcal{F}$ on $(S/\Sigma)_{cris}$, we define
$$
\mathcal{F}(n) := \lim_{e \geq N} \mathcal{F}( (S,\mathrm{Spec}(D(n)/p^e(n)),\delta(n))).
$$
Each $(S,\mathrm{Spec}(D(n)/p^e(n)),\delta(n))$ is simply the $(n+1)$-fold self-product of $(S,\mathrm{Spec}(D/p^e),\delta)$ in $(S/\Sigma)_{cris}$. Letting $n$ vary, we obtain a natural cosimplicial abelian group (or a cochain complex)
$$
\mathcal{F}(\bullet) := \big(\mathcal{F}(0) \to \mathcal{F}(1) \to \mathcal{F}(2) \cdots \big),
$$
that we call the {\em \v{C}ech-Alexander} complex of $\mathcal{F}$ associated to $D$. 
\end{notation}

\subsection{Some generalities on crystalline cohomology}

\noindent
This subsection collects certain basic tools necessary for working with crystalline cohomomology; these will be used consistently in the sequel. We begin with a brief review of the construction of homotopy-limits in the only context where they appear in this paper.

\begin{construction}
\label{construction-r-lim}
Let $\mathcal{C}$ be a topos. Fix a sequence $T_1 \subset T_2 \subset \cdots T_n \subset \cdots$ of monomorphisms in $\mathcal{C}$. We will construct the functor $R\lim_i R\Gamma(T_i,-)$; here we follow the convention that $\mathcal{G}(U) = \Gamma(U,\mathcal{G}) = \mathrm{Hom}_{\mathcal{C}}(U,\mathcal{G})$ for any pair of objects $U,\mathcal{G} \in \mathcal{C}$. Let $\mathrm{Ab}^{\mathbf{N}}$ denote the category of projective systems of abelian groups indexed by the natural numbers. The functor $\mathcal{F} \mapsto \lim_i \mathcal{F}(T_i)$ can be viewed as the composite
$$
\mathrm{Ab}(\mathcal{C}) \stackrel{ \{\Gamma(T_i,-)\}_i}{\to} \mathrm{Ab}^{\mathbf{N}} \stackrel{\lim_i}{\to} \mathrm{Ab}.
$$
Each of these functors is a left exact functor between abelian categories with enough injectives, so we obtain a composite of (triangulated) derived functors
$$
D^+(\mathrm{Ab}(\mathcal{C})) \stackrel{ \{R\Gamma(T_i,-)\}_i}{\to} D^+(\mathrm{Ab}^{\mathbf{N}}) \stackrel{R\lim_i}{\to} D^+(\mathrm{Ab}).
$$
We use $R\lim_i R\Gamma(T_i,-)$ to denote the composite functor. To identify this functor, observe that if we set $T = \colim_i T_i$, then $\mathcal{F}(T) = \lim_i \mathcal{F}(T_i)$ by adjunction. Moreover, for any injective object $\mathcal{I}$ of $\mathrm{Ab}(\mathcal{C})$, the projective system $i \mapsto \mathcal{I}(T_i)$ has surjective transition maps $\mathcal{I}(T_{i+1}) \to \mathcal{I}(T_i)$:  the maps $T_i \to T_{i+1}$ are injective, and $\mathcal{I}$ is an injective object. Since projective systems in $\mathrm{Ab}^{\mathbf{N}}$ with surjective transition maps are acyclic for the functor $\lim_i$ (by the Mittag-Leffler condition), there is an identification of triangulated functors
$$
R\Gamma(T,-) \simeq R\lim_i R\Gamma(T_i,-).
$$
Thus, the value $R\lim_i R\Gamma(T_i,\mathcal{F})$ is computed by $I^\bullet(T) = \lim_i I^\bullet(T_i)$, where $\mathcal{F} \to I^\bullet$ is an injective resolution. An observation that will be useful in the sequel is the following: if each $R\Gamma(T_i,\mathcal{F})$ is concentrated in degree $0$, then $R\lim_i R\Gamma(T_i,\mathcal{F})$ coincides with $R\lim_i \mathcal{F}(T_i)$, and thus has only two non-zero cohomology groups (as $R^j \lim_i A_i = 0$ for $j > 1$ and {\em any} $\mathbf{N}$-indexed projective system $\{A_i\}_i$ of abelian groups).
\end{construction}

\noindent
We use Construction \ref{construction-r-lim} to show that the \v{C}ech-Alexander complex often computes crystalline cohomology (compare with \cite[Theorem V.1.2.5]{Berthelot}).
\begin{lemma}
\label{lemma-cech}
Let $\mathcal{F}$ be a quasi-coherent $\mathcal{O}_{S/\Sigma}$-module. Assume that for each $n > 0$, the group $R^1 \lim_{e \geq N}$ 
vanishes for the projective system $e \mapsto \mathcal{F}((S, \mathrm{Spec}(D(n)/p^eD(n)), \delta(n)))$.
Then the complex $\mathcal{F}(\bullet)$ computes $R\Gamma(S/\Sigma, \mathcal{F})$.
\end{lemma}

\begin{proof}
As representable functors are sheaves on $(S/\Sigma)_{cris}$ (by Zariski descent), we freely identify objects of $(S/\Sigma)_{cris}$ with the corresponding sheaf on $(S/\Sigma)_{cris}$. One can easily check that the map $\colim_{e \geq N} \Big( (S,\mathrm{Spec}(D/p^e),\delta) \Big)\to \ast$ is an effective epimorphism
is the topos of sheaves on $(S/\Sigma)_{cris}$. Since filtered colimits and the Yoneda embedding both commute with finite products, the $(n+1)$-fold self-product of 
$\colim_{e \geq N} \Big( (S,\mathrm{Spec}(D/p^e),\delta) \Big)$ is simply $\colim_{e \geq N} \Big( (S,\mathrm{Spec}(D(n)/p^e(n)),\delta(n)) \Big)$. General topos theory (see Remark \ref{remark-cech-topos}) shows that $R\Gamma(S/\Sigma,\mathcal{F})$ is computed  by
$$
R\Gamma(\colim_{e \geq N} \Big( (S,\mathrm{Spec}(D(0)/p^e(0)),\delta(0)) \Big),\mathcal{F}) \to R\Gamma(\colim_{e \geq N} \Big( (S,\mathrm{Spec}(D(1)/p^e(1)),\delta(1)) \Big),\mathcal{F}) \to \cdots.
$$
The discussion in Construction \ref{construction-r-lim} and the vanishing of quasi-coherent sheaf cohomology on affine schemes then identify the above bicomplex with the bicomplex
$$
R\lim_{e \geq N} \big( \mathcal{F}((S, \mathrm{Spec}(D(0)/p^eD(0)), \delta(0)))\big) \to
R\lim_{e \geq N}  \big(\mathcal{F}((S, \mathrm{Spec}(D(1)/p^eD(1)), \delta(1)))\big) \to \cdots.
$$
The $R^1 \lim_e$ vanishing hypothesis ensures that the bicomplex above collapses to $\mathcal{F}(\bullet)$  proving the claim.
\end{proof}

\begin{remark}
\label{remark-cech-topos}
The following fact was used in the proof of Lemma \ref{lemma-cech}: if $\mathcal{C}$ is a topos, and $X \to \ast$ is an effective epimorphism in $\mathcal{C}$, then for any abelian sheaf $\mathcal{F}$ in $\mathcal{C}$, the object $R\Gamma(\ast,\mathcal{F})$ is computed by a bicomplex
$$
R\Gamma(X,\mathcal{F}) \to R\Gamma(X \times X, \mathcal{F}) \to R\Gamma(X \times X \times X,\mathcal{F}) \to \cdots,
$$
i.e., the choice of an injective resolution $\mathcal{F} \to I^\bullet$ defines a bicomplex $I^\bullet(X^{\bullet+1})$ whose totalisation computes $R\Gamma(\ast,\mathcal{F})$; this follows from cohomological descent since the augmented simplicial object $X^\bullet \to \ast$ is a hypercover. In particular, there is a spectral sequence with $E_1$-term given by $H^q(X^p,\mathcal{F})$ that converges to $H^{p+q}(\ast,\mathcal{F})$.
\end{remark}

\begin{remark}
The $R^1 \lim_{e \geq N}$ vanishing assumption of Lemma \ref{lemma-cech} will hold for all sheaves appearing in this paper. For quasi-coherent {\em crystals} $\mathcal{F}$, this assumption clearly holds as 
$$
\mathcal{F}( (S,\mathrm{Spec}(D(n)/p^e(n)),\delta(n)) ) \to \mathcal{F}( (S,\mathrm{Spec}(D(n)/p^{e-1}(n)),\delta(n)) )
$$ 
is surjective for all $e > N$ and all $n \geq 0$. By direct computation, the same is also true for the sheaves $\Omega^i_{S/\Sigma}$.
\end{remark}

\noindent
Next, we formulate and prove a purely algebraic lemma comparing $p$-adically complete $\mathbf{Z}_p$-modules with compatible systems of $\mathbf{Z}/p^e$-modules; the result is elementary and well-known, but recorded here for convenience. We remind the reader that a $\mathbf{Z}_p$-module $M$ is said to be {\em $p$-adically complete} if the natural map $M \to \lim_e M/p^e$ is an isomorphism. 

\begin{lemma}
\label{lemma-p-adic-limit}
The functor $M \mapsto (M/p^e,\mathrm{can}_e:\big(M/p^{e+1})/p^e \simeq M/p^e)$ defines an equivalence between the category of $p$-adically complete $\mathbf{Z}_p$-modules $M$ and the category of projective systems $(M_e,\phi_e)$ (indexed by $e \in \mathbf{N}$) with $M_e$ a $\mathbf{Z}/p^e$-module, and $\phi_e:M_{e+1}/p^e \simeq M_e$ an isomorphism.
\end{lemma}
\begin{proof}
A left-inverse functor is given by $(M_e,\phi_e) \mapsto M := \lim_e M_e$, the limit being taken along the maps $\phi_e$. To check that this is also a right inverse, it suffices to show that $(\lim_e M_e)/p^n \simeq M_n$ for any system $(M_e,\phi_e)$ as in the lemma. Projection defines a natural map $(\lim_e M_e)/p^n \to M_n$ which is surjective as the $\phi_e$'s are all surjective. For injectivity, it suffices to show that any element $m = (m_e) \in \lim_e M_e$ with $m_n = 0$ is divisible by $p^n$ in $\lim_e M_e$. The hypothesis implies that there exists an $m' \in \lim_e M_e$ such that $m - p^n m'$ maps to $0$ in $M_{n+1}$: we can simply take $m'$ to be an arbitrary lift of an element in $M_{n+1}$ which gives $m_{n+1}$ on multiplication by $p^n$ (which exists since $\phi_{e+1}$ maps $M_{n+1}/p^n$ isomorphically onto $M_n$). Continuing this process, for each $i > 0$, we can find an element $m_i \in \lim_e M_e$ such that $m - p^n m_i$ maps to $0$ in $M_{n+i}$. Taking the limit $i \to \infty$ proves the desired claim.
\end{proof}

\noindent
The next lemma is a standard result in crystalline cohomology (see \cite[Chapter IV]{Berthelot} and \cite[Theorem 6.6]{BerthelotOgusBook}). We
sketch the proof to convince the reader that this result is elementary.

\begin{lemma}
\label{lemma-category-crystals-quasi-coherent}
The category of crystals in quasi-coherent $\mathcal{O}_{S/\Sigma}$-modules
is equivalent to the category of pairs $(M, \nabla)$ where $M$ is a
$p$-adically complete $D$-module and
$\nabla : M \to M \otimes^\wedge_D \Omega^1_D$ is a topologically
quasi-nilpotent integrable connection.
\end{lemma}

\begin{proof}
Given a crystal in quasi-coherent modules $\mathcal{F}$ we set
$$
M = \mathcal{F}(0) := \lim_{e \geq N} \mathcal{F}((S, \mathrm{Spec}(D/p^eD), \delta))
$$
and $\nabla$ is as in (\ref{equation-connection}). Conversely,
suppose that $(M, \nabla)$ is a module with connection as in the statement
of the lemma. Then, given an affine object $(S \hookrightarrow T, \delta)$ of the
crystalline site corresponding to the divided power thickening
$(B \to A, \delta)$, we set
$$
\mathcal{F}((S, T, \delta)) = M \otimes_D B
$$
where $D \to B$ is any divided power map lifting $\text{id}_A : A \to A$.
Note that $p^mB = 0$ for some $m \geq 0$ by the definition of the crystalline 
site, so completion isn't needed in the formula. To see that this is well
defined suppose that $\varphi_1, \varphi_2 : D \to B$ are two maps
lifting $\text{id}_A$. Then we have an isomorphism
$$
M_{\phi_1,\phi_2}:M \otimes_{D, \varphi_1} B \longrightarrow M \otimes_{D, \varphi_2} B
$$
which is $B$-linear and characterized by the (Taylor) formula
$$
m \otimes 1 \longmapsto
\sum\nolimits_{E = (e_i)}
\left(\prod (\nabla_{\vartheta_i})^{e_i}\right)(m) \otimes
\prod \delta_{e_i}(h_i)
$$
where the sum is over all multi-indices $E$ with finite support.
The notation here is: $P = \mathbf{Z}_p[\{x_i\}_{i \in I}]$,
$\vartheta_i = \partial/\partial x_i$ and
$h_i = \varphi_2(x_i) - \varphi_1(x_i)$.
Since $h_i \in \text{Ker}(B \to A)$ it makes sense to apply the divided
powers $\delta_e$ to $h_i$.
The sum converges precisely because the connection is topologically
quasi-nilpotent (this can be taken as the definition). For three maps
$\varphi_1,\varphi_2,\varphi_3:D \to B$ lifting $\text{id}_A$, the resulting isomorphisms 
satisfy the cocycle condition
$$
M_{\phi_2,\phi_3} \circ M_{\phi_1,\phi_2} = M_{\phi_1,\phi_3}
$$
by the flatness of $\nabla$. Hence, 
the above recipe defines a sheaf on $(S/\Sigma)_{cris}$.
\end{proof}

\begin{remark}
\label{remark-replace-by-smooth-category-crystals}
Lemma \ref{lemma-category-crystals-quasi-coherent} remains valid if we replace the polynomial algebra $P$ appearing in Notation \ref{notation-affine} with any {\em smooth} $\mathbf{Z}_p$-algebra $P$ equipped with a surjection to $A$ (and $D$ with the corresponding $p$-adically completed divided power envelope). The only non-obvious point is to find a replacement for the Taylor series appearing in the formula for $M_{\phi_1,\phi_2}$ in the proof of Lemma \ref{lemma-category-crystals-quasi-coherent}. However, note first that the Taylor series makes sense as soon as there is a polynomial algebra $F$ and an \'etale map $F \to P$. Moreover, a ``change of variables'' computation shows that the resulting map is independent of choice of \'etale chart $F \to P$. The general case then follows by Zariski glueing.
\end{remark}

\noindent
We use Lemma \ref{lemma-category-crystals-quasi-coherent} to show that crystals on $(S/\Sigma)_{cris}$ are determined by their restriction to the special fibre.

\begin{corollary}
\label{corollary-reduce-modulo-p}
Reduction modulo $p$ gives an equivalence between categories of crystals of quasi-coherent modules over the structure sheaves of $(S/\Sigma)_{cris}$  and $(S \otimes_{\Sigma} \mathrm{Spec}(\mathbf{F}_p)/\Sigma)_{cris}$ 
\end{corollary}
\begin{proof}
This follows from Lemma \ref{lemma-category-crystals-quasi-coherent} by identifying, for any $e \geq N$, the divided power envelopes of $P/p^e \to A$ and $P/p^e \to A \to A/p$ (by \cite[Proposition I.2.8.2]{Berthelot}).
\end{proof}

\subsection{The main theorem in the affine case}

\noindent
The goal of this section is to prove the following theorem.

\begin{theorem}
\label{theorem-compute-by-de-Rham}
Suppose that $\mathcal{F}$ corresponds to a pair $(M, \nabla)$
as in Lemma \ref{lemma-category-crystals-quasi-coherent}. Then there is a natural quasi-isomorphism
$$
R\Gamma(S/\Sigma, \mathcal{F})
\simeq
(M \to M \otimes^\wedge_D \Omega^1_D \to
M \otimes^\wedge_D \Omega^2_D \to \cdots).
$$
\end{theorem}

\noindent
For $\mathcal{F}$ and $M$ as in Theorem \ref{theorem-compute-by-de-Rham}, we use $M(n)$ and 
$M(\bullet)$ instead of $\mathcal{F}(n)$ and $\mathcal{F}(\bullet)$ from Notation \ref{notation-affine}. Each $M(n)$ is
a $D(n)$-module with integrable connection as in \eqref{equation-connection},  so it defines a de Rham complex
$$
M(n) \to M(n) \otimes^\wedge_{D(n)} \Omega^1_{D(n)} \to
M(n) \otimes^\wedge_{D(n)} \Omega^2_{D(n)} \to \cdots.
$$
As $n$ varies, these complexes fit together to define a bicomplex, which we call the de Rham complex of $M(\bullet)$.  Our proof of Theorem \ref{theorem-compute-by-de-Rham} hinges on the observation that each side of the quasi-isomorphism occurring in the statement of Theorem \ref{theorem-compute-by-de-Rham} also appears in the de Rham complex of $M(\bullet)$: the left side is the $0$-th row, while the right side is the $0$-th column. Thus, the proof of Theorem \ref{theorem-compute-by-de-Rham} is reduced to certain acyclicity results for the de Rham complex of $M(\bullet)$, which we show next. The following lemma shows that the ``columns'' of this bicomplex are all quasi-isomorphic.

\begin{lemma}
\label{lemma-poincare}
The map of complexes
$$
M \otimes^\wedge_D \Omega^*_D \to M(n) \otimes^\wedge_{D(n)} \Omega^*_{D(n)}
$$
induced by any of the natural maps $D \to D(n)$ is a quasi-isomorphism.
\end{lemma}

\begin{proof}
This is the ``naive'' Poincare lemma. More precisely, each natural map $D \to D(n)$ defines an isomorphism of $D(n)$-modules $M(n) \simeq M \otimes_D^\wedge D(n)$ compatible with $\nabla$ by the crystalline nature of $\mathcal{F}$. Thus, there is a filtration
of $M(n) \otimes^\wedge_{D(n)} \Omega^*_{D(n)}$ whose graded pieces are
$M \otimes^\wedge_D \Omega^i_D \otimes^\wedge_D \Omega^*_{D(n)/D}$. Thus it suffices to
show the natural map 
$$
D \to \Big(D(n) \to \Omega^1_{D(n)/D} \to \Omega^2_{D(n)/D} \to \cdots\Big)
$$
is a quasi-isomorphism. This can be checked explicitly as $D(n)$ is a divided power polynomial
algebra over $D$ (see \cite[Lemma V.2.1.2]{Berthelot}).
\end{proof}

\noindent
Next, we identify the ``rows'' of the de Rham complex of $M(\bullet)$.

\begin{lemma}
\label{lemma-cech-improved}
The complex
$$
M \otimes^\wedge_D \Omega^i_D \to
M(1) \otimes^\wedge_{D(1)} \Omega^i_{D(1)} \to
M(2) \otimes^\wedge_{D(2)} \Omega^i_{D(2)} \cdots
$$
computes $R\Gamma(S/\Sigma, \mathcal{F} \otimes_{\mathcal{O}_{S/\Sigma}} \Omega^i_{S/\Sigma})$.
\end{lemma}

\begin{proof}
By Lemma \ref{lemma-p-adic-limit}, we have
$$
M \otimes^\wedge_{D} \Omega^i_{D}  \simeq \lim_{e \geq N} \Big(M/p^e \otimes_{D/p^eD} \Omega^i_{D/p^eD}\Big) \simeq \lim_{e \geq N} \Big( \big(\mathcal{F} \otimes_{\mathcal{O}_{S/\Sigma}} \Omega^i_{S/\Sigma} \big)(  (S,\mathrm{Spec}(D/p^e),\delta)) \Big),
$$
and similarly for the terms over $D(n)$. The claim now follows from
Lemma \ref{lemma-cech}, the fact that $\Omega^i_{S/\Sigma}$ is quasi-coherent, and the fact that the transition maps
$M/p^{e + 1}M \otimes_{D/p^{e+ 1}D} \Omega^i_{D/p^{e + 1}D} \to M/p^eM \otimes_{D/p^eD} \Omega^i_{D/p^eD}$ are surjective.
\end{proof}

\noindent
To finish the proof of Theorem \ref{theorem-compute-by-de-Rham}, we need an acyclicity result about the ``rows'' of the de Rham complex of $M(\bullet)$. First, we handle the case $M = D$, i.e., when $\mathcal{F} = \mathcal{O}_{S/\Sigma}$.

\begin{lemma}
\label{lemma-vanishing-omega-1}
The complex
$$
\Omega^1_D \to \Omega^1_{D(1)} \to \Omega^1_{D(2)} \to \cdots
$$
is homotopic to zero as a $D(\bullet)$-cosimplicial module.
\end{lemma}

\begin{proof}
This complex is equal to the base change of the cosimplicial module
$$
M_* = \left(
\Omega^1_P \to
\Omega^1_{P \otimes P} \to
\Omega^1_{P \otimes P \otimes P} \to \cdots
\right)
$$
via the cosimplicial ring map $P^{\otimes n + 1} \to D(n)$.
Hence it suffices to show that the cosimplicial module $M_*$
is homotopic to zero. Let $P = \mathbf{Z}_p[\{x_i\}_{i \in I}]$.
Then $P^{\otimes n + 1}$ is the polynomial algebra on the elements
$$
x_i(e) = 1 \otimes \cdots \otimes x_i \otimes \cdots \otimes 1
$$
with $x_i$ in the $e$th slot. The modules of the complex are free on the
generators $\text{d}x_i(e)$. Note that if $f : [n] \to [m]$ is a
map then we see that
$$
M_*(f)(\text{d}x_i(e)) = \text{d}x_i(f(e))
$$
Hence we see that $M_*$ is a direct sum of copies of
Example \ref{example-cosimplicial-module} indexed by $I$, and we win.
\end{proof}

\begin{example}
\label{example-cosimplicial-module}
Suppose that $A_*$ is any cosimplicial ring.
Consider the cosimplicial module $M_*$ defined by the rule
$$
M_n = \bigoplus\nolimits_{i = 0, ..., n} A_n e_i
$$
For a map $f : [n] \to [m]$ define $M_*(f) : M_n \to M_m$
to be the unique $A_*(f)$-linear map which maps $e_i$ to $e_{f(i)}$.
We claim the identity on $M_*$ is homotopic to $0$.
Namely, a homotopy is given by a map of cosimplicial modules
$$
h : M_* \longrightarrow \text{Hom}(\Delta[1], M_*)
$$
where $\Delta[1]$ denote the simplicial set whose set of $n$-simplices
is $\text{Mor}([n], [1])$, see \cite{Meyer}.
Let $\alpha^n_j : [n] \to [1]$ be defined by
$\alpha^n_j(i) = 0 \Leftrightarrow i < j$. Then we define $h$
in degree $n$ by the rule
$$
h_n(e_i)(\alpha^n_j) =
\left\{
\begin{matrix}
e_{i} & \text{if} & i < j \\
0 & \text{else} 
\end{matrix}
\right.
$$
We first check $h$ is a morphism of cosimplicial modules. Namely, for
$f : [n] \to [m]$ we will show that
\begin{equation}
\label{equation-cosimplicial-morphism}
h_m \circ M_*(f) = \text{Hom}(\Delta[1], M_*)(f) \circ h_n
\end{equation}
This is equivalent to saying that the left hand side of 
(\ref{equation-cosimplicial-morphism}) evaluted at $e_i$ is given by
$$
h(e_{f(i)})(\alpha^m_j) =
\left\{
\begin{matrix}
e_{f(i)} & \text{if} & f(i) < j \\
0 & \text{else}
\end{matrix}
\right.
$$
Note that $\alpha^n_j \circ f = \alpha^m_{j'}$ where
$0 \leq j' \leq n$ is such that $f(a) < j$ if and only if $a < j'$.
Thus the right hand side of
(\ref{equation-cosimplicial-morphism}) evaluted at $e_i$ is given by
$$
M_*(f)(h(e_i)(\alpha^m_j \circ f) =
M_*(f)(h(e_i)(\alpha^n_{j'})) =
\left\{
\begin{matrix}
e_{f(i)} & \text{if} & i < j' \\
0 & \text{else} 
\end{matrix}
\right.
$$
It follows from our description of $j'$ that the two answers are equal.
Hence $h$ is a map of cosimplicial modules.
Let $0 : \Delta[0] \to \Delta[1]$ and
$1 : \Delta[0] \to \Delta[1]$ be the obvious maps, and denote
$ev_0, ev_1 : \text{Hom}(\Delta[1], M_*) \to M_*$ the corresponding
evaluation maps. The reader verifies readily that the
the compositions
$$
ev_0 \circ h, ev_1 \circ h : M_* \longrightarrow M_*
$$
are $0$ and $1$ respectively, whence $h$ is the desired homotopy between
$0$ and $1$.
\end{example}

\noindent
We now extend Lemma \ref{lemma-vanishing-omega-1} to allow non-trivial coefficients.

\begin{lemma}
\label{lemma-affine-done}
For all $i > 0$ cosimplicial module
$$
M \otimes^\wedge_D \Omega^i_D \to
M(1) \otimes^\wedge_{D(1)} \Omega^i_{D(1)} \to
M(2) \otimes^\wedge_{D(2)} \Omega^i_{D(2)} \cdots
$$
is homotopy equivalent to zero.
\end{lemma}

\begin{proof}
The cosimplicial $D(\bullet)$-module above is  a (termwise) completed tensor 
product of the cosimplicial $D(\bullet)$-modules $M(\bullet)$ and $\Omega^i_{D(\bullet)}$.  Lemma \ref{lemma-vanishing-omega-1} shows that $\Omega^1_{D(\bullet)}$
is homotopy equivalent to zero as a cosimplicial $D(\bullet)$-module. The claim now follows as the following three operations preserve the property of being homotopy equivalent to $0$ for cosimplicial $D(\bullet)$-modules: termwise application of $\wedge^i$, tensoring with another cosimplicial $D(\bullet)$-module, and termwise $p$-adic completion.
\end{proof}

\noindent
The material above gives a rather pleasing proof that
crystalline cohomology is computed by the de Rham complex:

\begin{proof}[Proof of Theorem \ref{theorem-compute-by-de-Rham}]
We look at the first quadrant double complex $M^{\bullet, \bullet}$ with terms
$$
M^{n, m} = M(n) \otimes^\wedge_{D(n)} \Omega^m_{D(n)}.
$$
The horizontal differentials are given determined by the {\v C}ech-Alexander complex, while 
the vertical ones are given by the de Rham complex. By Lemma \ref{lemma-poincare}, 
each column complex $M^{n, \bullet}$ is quasi-isomorphic to de Rham complex 
$M \otimes^\wedge_D \Omega^\bullet_D$. Hence $H^m(M^{n, \bullet})$ is independent of 
$n$ and the differentials are
$$
H^m(M^{0, \bullet}) \xrightarrow{0}
H^m(M^{1, \bullet}) \xrightarrow{1}
H^m(M^{2, \bullet}) \xrightarrow{0}
H^m(M^{3, \bullet}) \xrightarrow{1} \cdots 
$$
We conclude that $\text{Tot}(M^{\bullet, \bullet})$ computes
the cohomology of the de Rham complex $M \otimes^\wedge_D \Omega^*_D$
by the first spectral sequence associated to the double complex.
On the other hand, Lemma \ref{lemma-cech} shows that
the ``row'' complex $M^{\bullet, 0}$ computes the cohomology of $\mathcal{F}$.
Hence if we can show that each $M^{\bullet, m}$ for $m > 0$ is acyclic, then 
we're done by the second spectral sequence. The desired vanishing now follows from 
Lemma \ref{lemma-affine-done}.
\end{proof}

\begin{remark}
\label{remark-replace-by-smooth-main-affine-theorem}
Lemma \ref{lemma-category-crystals-quasi-coherent} and Theorem \ref{theorem-compute-by-de-Rham} remain valid when $D$ is taken to be the $p$-adic completion of the divided power envelope of a surjection $P \to A$ with $P$ any {\em smooth} $\mathbf{Z}_p$-algebra. For Lemma \ref{lemma-category-crystals-quasi-coherent}, this was discussed in Remark \ref{remark-replace-by-smooth-category-crystals}. Thus, the only non-obvious point now is whether an analogue of Lemma \ref{lemma-vanishing-omega-1} is valid. However, at least Zariski locally on $\mathrm{Spec}(P)$, there is an \'etale map $F \to P$ with $F$ a polynomial $\mathbf{Z}_p$-algebra. Thus, the cotangent bundle of $P$ (and hence that of $D$) is obtained by base change from that of $F$,  so the required claim follows from the proof of Lemma \ref{lemma-vanishing-omega-1}. This shows that the assertion of Theorem \ref{theorem-compute-by-de-Rham} is true Zariski locally on $S$, and hence globally by the \v{C}ech spectral sequence for a suitable affine cover.
\end{remark}

\begin{remark}
\label{remark-affine-mod-p^e}
Let $\Sigma_e = \mathrm{Spec}(\mathbf{Z}/p^e)$, and let $S$ be an affine $\Sigma_e$-scheme. One can define the crystalline site $(S/\Sigma_e)_{cris}$ and crystals in the obvious way. The arguments given in this section work {\em mutatis mutandis} to show that the cohomology $R\Gamma(S/\Sigma_e,\mathcal{F})$ of a crystal $\mathcal{F}$ of quasi-coherent $\mathcal{O}_{S/\Sigma_e}$-modules is computed by the de Rham complex
$$
M_e \to M_e \otimes_{D/p^e} \Omega^1_{D/p^e} \to M_e \otimes_{D/p^e} \Omega^2_{D/p^e} \cdots,
$$
where $D/p^e$ is as in Notation \ref{notation-affine}, and $M_e = \mathcal{F}( (S,\mathrm{Spec}(D/p^e),\delta) )$ is the $D/p^e$-module that is the value of the crystal $\mathcal{F}$ on $\mathrm{Spec}(D/p^e)$, equipped with the integrable connection as in \eqref{equation-connection}.
\end{remark}

\section{Global analogues}
\label{section-global}

\noindent
Our goal in this section is to prove a global analogue (Theorem \ref{theorem-global-cohomology}) of the results of \S \ref{section-affine}, and deduce some geometric consequences (Corollaries \ref{corollary-formal-functions} and \ref{corollary-base-change}). In order to do so, we first conceptualize the work done in \S \ref{section-affine} as a {\em vanishing} result on arbitrary schemes in Theorem \ref{theorem-global}; this formulation gives us direct access to certain globally defined maps, which are then used to effortlessly reduce global statements to local ones.

\subsection{A vanishing statement}
Our vanishing result is formulated terms of the ``change of topology'' map relating the crystalline site to the Zariski site, whose construction we recall first. Let $f:S \to \Sigma$ be a map with $p$ locally nilpotent on $S$. There is a morphism of ringed topoi
$$
u_{S/\Sigma}: \Big(\mathrm{Shv}((S/\Sigma)_{cris}),\mathcal{O}_{S/\Sigma}\Big) \to \Big( \mathrm{Shv}(S_{zar}),f^{-1} \mathcal{O}_\Sigma \Big)
$$
characterised by the formula
$$
u_{S/\Sigma}^{-1}(F)( (U,T,\delta)) = F(U)
$$ 
for any sheaf $F \in \mathrm{Shv}(S_{zar})$ and object $(U,T,\delta) \in (S/\Sigma)_{cris}$ (see \cite[\S III.3.2]{Berthelot}).  The associated pushforward $Ru_{S/\Sigma}:D(\mathcal{O}_{S/\Sigma}) \to D(f^{-1} \mathcal{O}_\Sigma)$ is a localised version of crystalline cohomology, i.e., for any $U \in S_{zar}$, we have
$$
R\Gamma(U,Ru_{S/\Sigma}(F)) \simeq R\Gamma( U/\Sigma, F);
$$
see \cite[Corollary III.3.2.4]{Berthelot} for the corresponding statement at the level of cohomology groups. With this language, our main result is the following somewhat surprising theorem.

\begin{theorem}
\label{theorem-global}
Let $S$ be a scheme over $\Sigma$ such that $p$ is locally nilpotent on $S$.
Let $\mathcal{F}$ be a crystal in quasi-coherent
$\mathcal{O}_{S/\Sigma}$-modules. The truncation map of complexes
$$
(\mathcal{F} \to
\mathcal{F} \otimes_{\mathcal{O}_{S/\Sigma}} \Omega^1_{S/\Sigma} \to
\mathcal{F} \otimes_{\mathcal{O}_{S/\Sigma}} \Omega^2_{S/\Sigma} \to \cdots)
\longrightarrow \mathcal{F}[0],
$$
while not a quasi-isomorphism, becomes a quasi-isomorphism after applying $Ru_{S/\Sigma}$. In fact, for any $i > 0$, we have 
$$Ru_{S/\Sigma}(\mathcal{F} \otimes_{\mathcal{O}_{S/\Sigma}} \Omega^i_{S/\Sigma}) = 0.$$ 
\end{theorem}

\begin{proof}
This follows from the vanishing of the cohomology of the sheaves
$\mathcal{F} \otimes_{\mathcal{O}_{S/\Sigma}} \Omega^i_{S/\Sigma}$ over affines for $i > 0$, see
Lemmas \ref{lemma-cech-improved} and \ref{lemma-affine-done}. 
\end{proof}

\begin{remark}
The {\em proof} of Theorem \ref{theorem-global} shows that the conclusion $Ru_{S/\Sigma}(\mathcal{F} \otimes_{\mathcal{O}_S} \Omega^i_{S/\Sigma}) = 0$ for $i > 0$ is true for any quasi-coherent $\mathcal{O}_{S/\Sigma}$-module $\mathcal{F}$ which, locally on $S$, satisfies the $R^1 \lim_{e \geq N}$ vanishing condition of Lemma \ref{lemma-cech}. This applies to the following non-crystals: $\Omega^i_{S/\Sigma}$ for all $i$, and any sheaf of the form $u_{S/\Sigma}^{-1} \mathcal{F}$, where $\mathcal{F}$ is a quasi-coherent $\mathcal{O}_S$-module on $S_{zar}$. In particular, it applies to the sheaf $u_{S/\Sigma}^{-1} \mathcal{O}_S$ defined by $u_{S/\Sigma}^{-1} \mathcal{O}_S( (U,T,\delta) ) = \Gamma(U,\mathcal{O}_U)$.
\end{remark}

\subsection{Global results}
We now explain how to deduce global consequences from Theorem \ref{theorem-global}, such as the identification of crystalline cohomology with de Rham cohomology. First, we establish notation used in this section.

\begin{notation}
\label{notation-global}
Let $S$ be a $\Sigma$-scheme such that $p^N = 0$ on $S$. Assume there is a closed immersion $i : S \to X$ of $\Sigma$-schemes with $X$ finitely presented and smooth over $\Sigma$. For each $e \geq N$, set $\mathcal{D}_e$ to be the divided power hull of the map $\mathcal{O}_X/p^e \to \mathcal{O}_S$. Each $\mathcal{D}_e$ is supported on $S$, and letting $e$ vary defines a $p$-adic formal scheme $T$ with underlying space $S$ and structure sheaf $\lim_{e \geq N} \mathcal{D}_e$. Moreover, the category of quasi-coherent $\mathcal{O}_T$-modules can be identified with the category of compatible systems of $\mathcal{D}_e$-modules on $S$ (with compatibilities as in Lemma \ref{lemma-p-adic-limit}), and we only use $T$ as a tool for talking about such compatible systems. The sheaves $\mathcal{D}_e$ define (honest) subschemes $T_e = \mathrm{Spec}(\mathcal{D}_e) \subset T$ containing $S$. The quasi-compactness of $S$ then gives objects $(S,T_e,\delta)$ of $(S/\Sigma)_{cris}$. Following our conventions, let $\Omega^i_{T_e}$ be the pullback of the corresponding sheaf on $X$, and set $\Omega^i_T$ to be the result of glueing the $\Omega^i_{T_e}$.
\end{notation}

\noindent
Let $S$, $T_e$, and $T$ be as in Notation \ref{notation-global}, and let $\mathcal{F}$ be a crystal in quasi-coherent $\mathcal{O}_{S/\Sigma}$-modules. Restricting $\mathcal{F}$ to $(S,T_e,\delta)$ defines quasi-coherent $\mathcal{O}_{T_e}$-modules $\mathcal{F}_{T_e}$ for $e \geq N$, and hence a quasi-coherent $\mathcal{O}_T$-module $\mathcal{M}$ by glueing. The integrable connections $\mathcal{F}_{T_e} \to \mathcal{F}_{T_e} \otimes_{\mathcal{O}_{T_e}} \Omega^1_{T_e}$ coming from \eqref{equation-connection} glue to give an integrable connection
$$
\nabla : \mathcal{M} \longrightarrow
\mathcal{M} \otimes_{\mathcal{O}_T} \Omega^1_T,
$$
which then defines a de Rham complex on $T$.

\begin{theorem}
\label{theorem-global-cohomology}
Let $S \to X$, $T$, $\mathcal{F}$, and $\mathcal{M}$ be as in Notation \ref{notation-global} and the following discussion. The hypercohomology on $T$ of the complex 
$$
\mathcal{M} \to
\mathcal{M} \otimes_{\mathcal{O}_T} \Omega^1_T \to 
\mathcal{M} \otimes_{\mathcal{O}_T} \Omega^2_T \to \cdots 
$$
computes $R\Gamma(S/\Sigma, \mathcal{F})$.
\end{theorem}

\begin{proof}
First, we construct the map. By basic formal scheme theory, we have a formula
\begin{eqnarray*}
R\lim_{e \geq N} R\Gamma( (S,T_e,\delta), \mathcal{F} \otimes_{\mathcal{O}_{S/\Sigma}} \Omega^i_{S/\Sigma}) &=& R\lim_{e \geq N} R\Gamma(T_e,\mathcal{F}_{T_e} \otimes_{\mathcal{O}_{T_e}} \Omega^i_{T_e})  \\
&=& R\Gamma(T,R\lim_{e \geq N} \big(\mathcal{F}_{T_e} \otimes_{\mathcal{O}_{T_e}} \Omega^i_{T_e} \big)) \\
&=&  R\Gamma(T,\mathcal{M} \otimes_{\mathcal{O}_T} \Omega^i_T)
\end{eqnarray*}
for each $i \geq 0$. Here the first equality follows from the definition of cohomology in the crystalline site; the second equality follows from the identification of the (derived functors of the) composite functors
$$
\mathrm{Ab}(\mathcal{C})^{\mathbf{N}} \stackrel{\Gamma(T,-)}{\to} \mathrm{Ab}^{\mathbf{N}} \stackrel{\lim}{\to} \mathrm{Ab} \quad \textrm{and} \quad
\mathrm{Ab}(\mathcal{C})^{\mathbf{N}} \stackrel{\lim}{\to} \mathrm{Ab}(\mathcal{C}) \stackrel{\Gamma(T,-)}{\to} \mathrm{Ab};
$$
the last equality follows from the vanishing of 
$$
R^j \lim_{e \geq N} \big(\mathcal{F}_e \otimes_{\mathcal{O}_{T_e}} \Omega^i_{T_e}\big)
$$
for all $j > 0$, which follows from the vanishing of higher quasi-coherent sheaf cohomology for affines. Using this formula, and applying 
$$
R\Gamma(*, -) \longrightarrow R\lim_{e \geq N} R\Gamma((S, T_e, \delta), -)
$$
to the morphism in Theorem \ref{theorem-global}, gives the desired map
$$
R\Gamma(S/\Sigma,\mathcal{F}) \to R\Gamma(T,\mathcal{M} \to \mathcal{M} \otimes_{\mathcal{O}_T} \Omega^1_T \to \mathcal{M} \otimes_{\mathcal{O}_T} \Omega^2_T \to \cdots).
$$
Moreover, this map is an isomorphism for affine $S$ by Theorem \ref{theorem-compute-by-de-Rham}  (see  Remark \ref{remark-replace-by-smooth-main-affine-theorem}), and is functorial in $S$. The \v{C}ech spectral sequence for an affine open cover then immediately implies the claim for $X$ is quasi-compact and separated (as the $E_2$ terms involve cohomology of affines). For an $X$ only assumed to be quasi-compact and quasi-separated, another application of the spectral sequence for an affine open cover finishes the proof (as the $E_2$ terms involve cohomology on quasi-compact and separated schemes). Since all smooth finitely presented $\Sigma$-schemes are quasi-compact and quasi-separated, we are done.
\end{proof}

\begin{remark}
The arguments that go into proving Theorem \ref{theorem-global-cohomology} also apply {\em mutatis mutandis} to reprove Grothendieck's comparison theorem from \cite{GrothendieckCrystals}: if $S$ is a variety over $\mathbf{C}$, and $S \subset X$ is a closed immersion into a smooth variety, and $T$ denotes the formal completion of $X$ along $S$, then the cohomology of the structure sheaf on the infinitesimal site $(S/\mathrm{Spec}(\mathbf{C}))_{inf}$ is computed by the hypercohomology on $S$ of the de Rham complex of $T$ (defined suitably). The only essential change is that the proof of Lemma \ref{lemma-poincare}, which relies on the vanishing of the higher de Rham cohomology of a divided power polynomial algebra, must be replaced by its formal analogue, i.e., the vanishing of the higher de Rham cohomology of a formal power series ring in characteristic $0$.
\end{remark}

\noindent
In certain situations, Theorem \ref{theorem-global-cohomology} can be algebraised to get a statement about classical schemes. For example:

\begin{corollary}
\label{corollary-formal-functions}
Let $f:X \to \Sigma$ be a proper smooth morphism, and set 
$S = X \times_{\Sigma} \mathrm{Spec}(\mathbf{F}_p)$. Then the hypercohomology
of the de Rham complex 
$$
\mathcal{O}_X \to
\Omega^1_{X/\Sigma} \to
\Omega^2_{X/\Sigma} \to \cdots,
$$
computes $R\Gamma(S/\Sigma,\mathcal{O}_{S/\Sigma})$. In particular, the de Rham 
cohomology of $X \to \Sigma$ is determined functorially by the fibre $S \hookrightarrow X$ of $f$, and thus admits a Frobenius action.
\end{corollary}
\begin{proof}
This follows from Theorem \ref{theorem-global-cohomology} and the formal functions theorem as $T$ is just the $p$-adic completion of $X$: the ideal $\mathrm{ker}(\mathcal{O}_X \to \mathcal{O}_S) = (p) \subset \mathcal{O}_X$ already has specified divided powers, so $\mathcal{O}_X/p^e = \mathcal{D}_e$ for any $e$.
\end{proof}

\begin{remark}
One can upgrade Corollary \ref{corollary-formal-functions} to a statement 
that incorporates coefficients as follows. There is an equivalence of 
categories between crystals in quasi-coherent sheaves on $(S/\Sigma)$ and quasi-coherent sheaves
on $X$ equipped with a flat connection relative to $\Sigma$ (see 
Lemma \ref{lemma-category-crystals-quasi-coherent} and Corollary \ref{corollary-reduce-modulo-p}). This equivalence respects 
cohomology, i.e., the crystalline cohomology of a crystal in quasi-coherent sheaves on 
$S/\Sigma$ is computed by the de Rham cohomology of the corresponding quasi-coherent sheaf 
with flat connection on $X$.
\end{remark}

\noindent
We conclude with a brief discussion of the base change behaviour. For an $\mathbf{Z}/p^N$-scheme $S$ as above, a natural question is whether crystalline cohomology relative to $\Sigma_e = \mathrm{Spec}(\mathbf{Z}/p^e)$ (for $e \geq N$) can be recovered from crystalline cohomology relative to $\Sigma$ via (derived) base change along $\mathbf{Z}_p \to \mathbf{Z}/p^e$. In general, the answer is ``no,'' see Example \ref{example-bo-torsion}. However, under suitable flatness conditions, the answer is ``yes'':

\begin{corollary}
\label{corollary-base-change}
Let $S \to X, T$ and $\mathcal{F}$ be as in Notation \ref{notation-global} and the following discussion.  Assume that for each $e \geq N$, the $\mathcal{O}_{T_e}$-module $\mathcal{F}_{T_e}$ is flat over $\mathbf{Z}/p^e$. Then for each $e \geq N$, we have a base change isomorphism
$$
\mathbf{Z}/p^e \otimes_{\mathbf{Z}}^L R\Gamma(S/\Sigma,\mathcal{F}) \simeq R\Gamma(S/\Sigma_e,\mathcal{F}|_{S/\Sigma_e}),
$$
where $\mathcal{F}|_{S/\Sigma_e}$ denotes the restriction of $\mathcal{F}$ along $(S/\Sigma_e)_{cris} \subset (S/\Sigma)_{cris}$.
\end{corollary}
\begin{proof}
By Theorem \ref{theorem-global-cohomology}, we have
$$
R\Gamma(S/\Sigma,\mathcal{F}) \simeq R\Gamma(T,
\mathcal{M} \to
\mathcal{M} \otimes_{\mathcal{O}_T} \Omega^1_T \to 
\mathcal{M} \otimes_{\mathcal{O}_T} \Omega^2_T \to \cdots).
$$
Since each $\Omega^i_T$ is flat over $\mathcal{O}_T$, the tensor products $\mathcal{M} \otimes_{\mathcal{O}_T} \Omega^i_T$ appearing in the de Rham complex above are automatically derived tensor products. Applying $\mathbf{Z}/p^e \otimes^L_{\mathbf{Z}} -$ (and observing that this operation commutes with applying $R\Gamma(T,-)$) then shows that $\mathbf{Z}/p^e \otimes^L_{\mathbf{Z}} R\Gamma(S/\Sigma, \mathcal{F})$ is computed as the hypercohomology on $T$ of the complex 
$$
K := \Big( \mathbf{Z}/p^e \otimes^L_{\mathbf{Z}} \mathcal{M}   \to
 \mathbf{Z}/p^e \otimes^L_{\mathbf{Z}} \mathcal{M} \otimes_{\mathcal{O}_T} \Omega^1_T \to 
 \mathbf{Z}/p^e \otimes^L_{\mathbf{Z}}  \mathcal{M} \otimes_{\mathcal{O}_T} \Omega^2_T \to \cdots\Big).
$$
The flatness assumption on $\mathcal{F}$ implies that $\mathcal{M}$ is flat over $\mathbf{Z}_p$. Lemma \ref{lemma-p-adic-limit} then shows that
$$
 \mathbf{Z}/p^e \otimes^L_{\mathbf{Z}} \mathcal{M} \simeq \mathbf{Z}/p^e \otimes_{\mathbf{Z}} \mathcal{M} \simeq \mathcal{F}_{T_e},
$$
and so
$$
 \mathbf{Z}/p^e \otimes^L_{\mathbf{Z}} \mathcal{M} \otimes_{\mathcal{O}_T} \Omega^i_T \simeq  \mathcal{F}_{T_e} \otimes_{\mathcal{O}_{T_e}} \Omega^i_{T_e}.
$$
In other words, the complex $K$ of sheaves appearing above is identified with the de Rham complex of the $\mathcal{O}_{T_e}$-module $\mathcal{F}_{T_e}$. The claim now follows from the modulo $p^e$ version of Theorem \ref{theorem-global-cohomology} (see also Remark \ref{remark-affine-mod-p^e}).
\end{proof}

\noindent
The hypotheses of Corollary \ref{corollary-base-change} are satisfied, for example, when $S$ is a flat local complete intersection over $\mathbf{Z}/p^N$, and $\mathcal{F}$ is a crystal in {\em locally free} quasi-coherent $\mathcal{O}_{S/\Sigma}$-modules; the smooth case is discussed in \cite[\S V.3.5]{Berthelot}. Moreover, there are extremely simple examples illustrating the sharpness of the lci assumption:

\begin{example}
\label{example-bo-torsion}
Let $S = \mathbf{F}_p[x,y]/(x^2,xy,y^2)$. In \cite[Appendix (A.2)]{BerthelotOgus}, Berthelot-Ogus exhibit a non-zero $p$-torsion class $\tau \in H^0(S/\Sigma,\mathcal{O}_{S/\Sigma})$ by constructing a non-zero $p$-torsion $\nabla$-horizontal element of the $p$-adically completed divided power envelope of the natural surjection $\mathbf{Z}_p[x,y] \to \mathbf{F}_p[x,y]/(x^2,xy,y^2)$. Via the exact triangle 
$$
R\Gamma(S/\Sigma,\mathcal{O}_{S/\Sigma}) \stackrel{p}{\to} R\Gamma(S/\Sigma,\mathcal{O}_{S/\Sigma}) \to R\Gamma(S/\Sigma,\mathcal{O}_{S/\Sigma}) \otimes^L_{\mathbf{Z}} \mathbf{Z}/p,
$$
$\tau$ defines a non-zero class in $H^{-1}\Big(R\Gamma(S/\Sigma,\mathcal{O}_{S/\Sigma}) \otimes^L_{\mathbf{Z}} \mathbf{Z}/p\Big)$. In particular, $R\Gamma(S/\Sigma,\mathcal{O}_{S/\Sigma}) \otimes^L_{\mathbf{Z}} \mathbf{Z}/p$ has cohomology in negative degrees, so it cannot be equivalent to $R\Gamma(S/\Sigma_1,\mathcal{O}_{S/\Sigma_1})$ (or the cohomology of any sheaf on any site).
\end{example}

\bibliography{my}
\bibliographystyle{amsalpha}

\end{document}